\documentclass[a4paper,12pt]{amsart}
\usepackage{amssymb,amsmath,a4wide,fullpage,setspace,microtype,textcomp,tikz,tikz-cd,accents,mathtools,abraces}
\usepackage{hyperref}
\definecolor{dark-red}{rgb}{0.5,0.15,0.15}
\hypersetup{
	colorlinks   = true, 
	urlcolor     = dark-red, 
	linkcolor    = black, 
	citecolor   = dark-red 
}
\usepackage[shortlabels]{enumitem}
\usepackage[T1]{fontenc}
\usepackage{lmodern}
\usepackage[numbers,sort&compress]{natbib}

\title{Regular clock maps and trace spaces}
\author{Philippe Gaucher}
\address{Universit\'e Paris Cit\'e, CNRS, IRIF, F-75013, Paris, France}
\urladdr{http://www.irif.fr/{\~{}}gaucher} 
\swapnumbers

\makeatletter
\let\leq\@undefined
\let\geq\@undefined
\let\top\@undefined
\let\vec\@undefined
\makeatother

\newtheorem*{thmN}{Theorem}
\newtheorem*{quN}{Question}
\newtheorem{thm}{Theorem}[section]
\newtheorem{prop}[thm]{Proposition}
\newtheorem{lem}[thm]{Lemma}
\newtheorem{cor}[thm]{Corollary}
\theoremstyle{definition}
\newtheorem{definition}[thm]{Definition}
\newtheorem{rem}[thm]{Remark}

\newtheorem{nota}[thm]{Notation}
\newcommand{\bd}{\begin{definition}}
\newcommand{\ed}{\end{definition}}

\newcommand{\leq}{\leqslant}
\newcommand{\geq}{\geqslant}
\newcommand{\Top}{\mathbf{Top}}
\newcommand{\ttop}{\mathbf{Top}}
\newcommand{\DTop}{\mathbf{dTop}}
\newcommand{\SDTop}{\mathbf{sdTop}}
\newcommand{\Clock}{\mathbf{Clock}}
\newcommand{\RegClock}{\mathbf{RegClock}}

\newcommand{\I}{\mathcal I}
\newcommand{\Istar}{\mathcal I_0}
\newcommand{\Path}{\vec P}
\newcommand{\Trace}{\vec T}
\newcommand{\Sdir}{\vec S^{1}}
\newcommand{\Circle}{S^1}
\newcommand{\Disk}{D^2}
\newcommand{\Id}{\operatorname{id}}
\newcommand{\abs}[1]{\lvert #1\rvert}

\newcommand{\vec}{\overrightarrow}

\begin{document}

\begin{abstract}
A regular clock map is a regular map of directed spaces from a saturated directed space to the directed circle. We prove that the category of regular clock maps is a small-orthogonality class in the category of clock maps. Hence it is locally presentable. The geometric realization of any precubical set or transverse set gives rise to a regular clock map. Finally, we prove that for the underlying directed space of a regular clock map, the canonical quotient from directed paths to traces is always a homotopy equivalence.
\end{abstract}

\maketitle
\tableofcontents
\hypersetup{linkcolor = dark-red}

\section{Introduction}

Directed spaces are topological spaces equipped with a distinguished set of continuous paths called directed paths that are closed under nondecreasing reparametrization and composition. These spaces serve as geometric models for concurrency, with examples arising in both cubical and globular settings. 

In the cubical setting, directed spaces are realized as geometric realizations of precubical sets, where directed paths are locally nondecreasing with respect to each coordinate axis \cite{DAT_book}. In the globular setting, they appear as realizations of cellular multipointed $d$-spaces, where directed paths are segments of execution paths \cite[Theorem~4.9]{GlobularNaturalSystem}.

Each directed path in a directed space $X$ can be associated with its trace, its equivalence class under nondecreasing endpoint-preserving reparametrization. In both settings, the canonical quotient map $\vec{P}(X)(\alpha,\beta) \to \vec{T}(X)(\alpha,\beta)$ from the space of directed paths (from $\alpha$ to $\beta$) to the corresponding trace space is known to be a homotopy equivalence: for the cubical case, see \cite[Proposition~2.16]{MR2521708}, and for the globular case, see \cite[Theorem~16]{Moore3}.

However, it is unknown whether this map is always a weak homotopy equivalence for arbitrary directed spaces. This raises a critical question: if $\vec{P}(X)(\alpha,\beta)$ and $\vec{T}(X)(\alpha,\beta)$ have different homotopy types, which one should be considered the correct one? 

We partially address the issue by introducing an additional structure to the axiomatic framework of directed spaces. Instead of considering directed spaces alone, we work with regular clock maps, namely regular maps of directed spaces from a saturated directed space to the directed circle. The results of this paper are the following:

\begin{thmN}  (Corollary~\ref{cor:regclock-lp} and Theorem~\ref{thm:main})
	The category $\RegClock$ of regular clock maps is locally presentable. For the saturated underlying directed space of any regular clock map, the canonical quotient from directed paths to traces is a homotopy equivalence for any pair of states.
\end{thmN}

Geometric realizations of precubical sets \cite{DAT_book} and transverse sets \cite{DirectedDegeneracy,ThickCubes} as directed spaces fit naturally into this framework, following \cite[Section~2.2.1]{MR2521708}: naturalization in Raussen's sense corresponds to clock-normalization (Definition~\ref{def:naturalization-directed-path} and Proposition~\ref{prop:normalization-continuous}). However, this approach only provides a partial solution, as the globular setting does not fit into the framework.

In \cite[Remark~2.17]{MR2521708}, Martin Raussen notes that a continuous additive length functional on the space of directed paths should enable us to recover the results of this paper. Indeed, equipping a saturated directed space with a continuous, additive, endpoint-reparametrization-invariant regular length function in the sense of Appendix~\ref{sec:stronger} suffices to show that the canonical quotient map from directed paths to their traces is a homotopy equivalence by arguments parallel to those of the present paper. Moreover, Definition~\ref{def:clock-length} confirms that such a length function always exists for the underlying directed space of any regular clock map. However, we do not pursue this approach, since the category of saturated directed spaces equipped with such a length function appears to lack convenient categorical properties. The clock map approach, by contrast, is more promising because it yields a locally presentable category. The local presentability raises the following question:  

\begin{quN}
	Let $\partial C_n=\partial\vec{[0,1]^n}\to \Sdir$ be the realization of the boundary of the $n$-cube as a regular clock map. Let $C_n=\vec{[0,1]^n}\to \Sdir$ be the realization of the $n$-cube as a regular clock map. The minimal model category structure on $\RegClock$ with respect to the set $\{\partial C_n \subseteq C_n\mid n\geq 0\}$ exists by \cite[Theorem~1.4]{henry2020minimal}. What is its geometric content ?
\end{quN} 

This model structure is particularly interesting because all geometric realizations of precubical sets and cofibrant transverse sets as directed spaces are cofibrant. Moreover, all morphisms of regular clock maps are regular, thereby preventing contraction of the directed segment---a critical point for preserving the causal structure.

The paper is organized as follows. Section~\ref{sec:reminders} recalls background material on directed spaces, saturated directed spaces, regularity, and local presentability. Section~\ref{sec:local-presentable-clock-map} introduces the notions of clock map and regular clock map, culminating with the proof that regular clock maps form a reflective and coreflective, locally presentable full subcategory of the category of clock maps. Section~\ref{sec:clock-lift-length} recalls the notion of trace space and introduces the clock length of a directed path. Section~\ref{sec:clock-normalization} introduces the key notion of clock normalization for directed paths. Section~\ref{sec:deformation-retract} establishes the necessary continuity lemmas, leading to the proof of the main theorem. Appendix~\ref{sec:stronger} explains why the regular clock map approach is strictly stronger than Raussen's length function approach.

\section{Directed spaces, saturation, and regularity}
\label{sec:reminders}

Let $\Top$ denote the cartesian closed category of $\Delta$-generated spaces or the cartesian closed category of $\Delta$-Hausdorff $\Delta$-generated topological spaces (cf. \cite[Section~2 and Appendix~B]{leftproperflow}). It is cartesian closed by a result due to Dugger and Vogt, recalled in \cite[Proposition~2.5]{mdtop}, and locally presentable by \cite[Corollary~3.7]{FR}. The right adjoint of the inclusion functor from $\Delta$-generated spaces to general topological spaces is called the $\Delta$-kelleyfication.  The internal hom, still denoted by $\ttop(-,-)$, is equipped with the $\Delta$-kelleyfication of the compact-open topology. Compact means quasicompact and Hausdorff.

The set $\I$ of nondecreasing continuous maps from $[0,1]$ to $[0,1]$ is equipped with the compact-open topology, which is $\Delta$-generated, being metrizable and locally path-connected. Let $\Istar=\{\phi\in\I\mid \phi(0)=0,\ \phi(1)=1\}$ be the subspace of maps in $\I$ fixing the endpoints.

Let $\gamma_1$ and $\gamma_2$ be two continuous maps from $[0,1]$ to some topological space such that $\gamma_1(1)=\gamma_2(0)$. The continuous map defined by 
\[
(\gamma_1 *_N \gamma_2)(t) = 
\begin{cases}
	\gamma_1(2t)& \hbox{ if }0\leq t\leq \frac{1}{2},\\
	\gamma_2(2t-1)& \hbox{ if }\frac{1}{2}\leq t\leq 1
\end{cases}
\]
is called the \textit{normalized composition}.

\bd \cite[Definition~1.1]{mg} \cite[Definition~4.1]{DAT_book} \label{def:directed_space} A \textit{directed space} is a pair $X=(\abs X,\vec{P}(X))$ consisting of a topological space $\abs X$ in $\Top$ and a set $\vec{P}(X)$ of continuous paths from $[0,1]$ to $\abs X$ satisfying the following axioms:
\begin{itemize}
	\item $\Path(X)$ contains all constant paths;
	\item $\Path(X)$ is closed under normalized composition;
	\item $\Path(X)$ is closed under reparametrization by an element of $\mathcal{I}$.
\end{itemize}
The set $\Path(X)$ is called a \textit{directed structure}. The space $\abs X$ is called the \textit{underlying topological space} or the \textit{state space}. The elements of $\Path(X)$ are called \textit{directed paths}. A morphism of directed spaces is a continuous map $f:\abs X\to |Y|$ such that for every directed path $\gamma$ of $X$, the composition $f\circ \gamma$ is a directed path of $Y$. The category of directed spaces is denoted by $\DTop$.
\ed  

The set $\Path(X)$ is equipped with the $\Delta$-kelleyfication of the compact-open topology. Let $\Path(X)(\alpha,\beta)$ denote the subspace of $\vec{P}(X)$ of directed paths $u$ such that $u(0)=\alpha$ and $u(1)=\beta$. 

The category $\DTop$ is locally presentable (see \cite[Theorem~4.2]{FR} and the remark before \cite[Proposition~3.6]{GlobularNaturalSystem}). By \cite[Remark~4.3~(ii)]{FR}, the set of directed paths of a colimit of directed spaces is given by a final structure: it is obtained by taking the closure under concatenation and reparametrization. Thus, the functor from $\DTop$ to $\Top$ taking a directed space $X=(\abs X,\Path(X))$ to its state space $\abs X$ is topological in the sense of \cite[Chapter~VI]{topologicalcat}.

\bd
The \textit{directed interval} $\vec I$ is $[0,1]$ equipped with the directed paths which are nondecreasing continuous maps $[0,1] \to [0,1]$. The \textit{directed circle} is
\[
\Sdir=\mathbb R/\mathbb Z
\]
equipped with the directed paths $c:[0,1]\to\mathbb R/\mathbb Z$ which admit a nondecreasing lift $\widetilde c:[0,1]\to\mathbb R$.  The quotient map $\pi:\mathbb R\to\mathbb R/\mathbb Z$ will be fixed throughout.
\ed

\bd\label{def:saturated} \cite[Definition~4.3]{reparam} \cite[Remark~4.3]{FR}  \cite[Definition~2.9]{zbMATH06404305}
A directed space $X$ is \textit{saturated} if the following condition holds: whenever $v:[0,1]\to\abs X$ is continuous, $\phi\in\Istar$, and $v\circ\phi\in\Path(X)$, then $v\in\Path(X)$.
\ed

Thus saturation says that if an endpoint-preserving slowdown of a path is directed, then the original path is already directed.  We denote by $\SDTop$ the full subcategory of directed spaces spanned by saturated directed spaces.

In \cite{Ziemiaski2012}, this notion leads to an isomorphism between the category of saturated directed spaces and the full subcategory of good streams of the category of Krishnan's streams in the sense of \cite{MR2545830}.

\begin{thm}\label{thm:sat-lp}
	The category $\SDTop$ of saturated directed spaces is locally presentable.
\end{thm}

\begin{proof}
This argument follows the approach outlined in \cite[Proposition~3.6]{GlobularNaturalSystem}. The necessary material is already developed in \cite{FR}. It suffices to start from a small relational universal strict Horn theory $\mathcal{T}$ axiomatizing $\Top$:
\begin{itemize}
	\item without equality, as shown in \cite[Theorem~3.6]{FR}, when $\Top$ is the category of $\Delta$-generated spaces;
	\item with equality, as shown in \cite[Proposition~B.18]{leftproperflow}, when $\Top$ is the category of $\Delta$-Hausdorff $\Delta$-generated spaces.
\end{itemize}
The notion of directed space is axiomatized using the axioms presented in the proof of \cite[Theorem~4.2]{FR}, together with the saturation hypothesis axiomatized in \cite[Remark~4.3]{FR}. The proof is completed using \cite[Theorem~5.30]{TheBook}.
\end{proof}

Because $\SDTop$ is defined by $\mathcal{T}$ plus a small relational universal strict Horn theory without equality, the forgetful functor from $\SDTop$ to $\Top$, which maps a directed space $X = (\abs{X}, \Path(X))$ to its underlying space $\abs{X}$, is also topological.

Both $\vec I$ and $\Sdir$ are saturated directed spaces.  For $\vec I$, this follows from the fact that if $v\phi$ is nondecreasing and $\phi\in\Istar$, then $v$ is nondecreasing because $\phi$ is surjective. For $\Sdir$, one uses the same argument after lifting directed paths to $\mathbb R$.

\begin{nota} \cite[Definition~1.32]{TheBook}
	If $\mathcal C$ is a category and $\mathcal M$ is a class of morphisms of $\mathcal C$, then
	\[
	\mathcal M^{\perp}
	\]
	denotes the full subcategory of objects $C\in\mathcal C$ such that for every $m:A\to B$ in $\mathcal M$, the precomposition map
	\[
	\mathcal C(B,C)\longrightarrow \mathcal C(A,C)
	\]
	is a bijection.  If $\mathcal M$ is a set, $\mathcal M^{\perp}$ is called a \textit{small-orthogonality class}.
\end{nota}  

It is easy to see that $\SDTop$ is a small-orthogonality class with respect to the set of maps $\{{\vec I}_\phi \subset {\vec I}\mid \phi\in \Istar\}$ where ${\vec I}_\phi$ is the segment $[0,1]$ equipped with the smallest directed structure containing $\phi$ as a directed path: 
\[
\SDTop = \{{\vec I}_\phi \subset {\vec I}\mid \phi\in \Istar\}^\perp.
\]

\bd \cite[Definition~1.1]{reparam}
Let $X$ be a directed space.  A directed path $u\in\Path(X)$ is called \textit{regular} if, for every nondegenerate interval $[a,b]\subseteq[0,1]$, the restriction $u|_{[a,b]}$ is not constant.
\ed

Note that, unlike in \cite[Definition~1.1]{reparam}, constant paths are not regular in this paper.

\bd\label{def:regular-map}
A morphism $f:X\to Y$ of directed spaces is \textit{regular} if, for every $u\in\Path(X)$,
\[
  f\circ u \text{ constant}
  \quad\Longrightarrow\quad
  u \text{ constant}.
\]
Equivalently, $f$ does not collapse a nonconstant directed path to a point.
\ed

\begin{lem}\label{lem:regular-local}
Let $f:X\to Y$ be a regular map of directed spaces.  If
$u\in\Path(X)$ and $[a,b]\subseteq[0,1]$, then
\[
  f\circ u \text{ constant on }[a,b]
  \quad\Longrightarrow\quad
  u \text{ constant on }[a,b].
\]
In particular, if $u$ is a regular directed path, then $f\circ u$ is a regular directed path.
\end{lem}

\begin{proof}
If $a=b$, there is nothing to prove.  If $a<b$, let $\lambda_{a,b}:[0,1]\to[0,1]$ be the affine map $\lambda_{a,b}(t)=a+(b-a)t$.  Then $u\circ\lambda_{a,b}$ is directed, because directed paths are closed under reparametrization by elements of $\I$.  If $f\circ u$ is constant on $[a,b]$, then $f\circ u\circ\lambda_{a,b}$ is constant.  Regularity of $f$ implies that $u\circ\lambda_{a,b}$ is constant, which is exactly the assertion that $u$ is constant on $[a,b]$.
\end{proof}

\section{Clock maps, regular clock maps, and local presentability}
\label{sec:local-presentable-clock-map}

\bd
The category of \textit{clock maps} is the slice category
\[
  \Clock=\SDTop/\Sdir.
\]
Thus an object is a morphism of saturated directed spaces
\[
  p:X\longrightarrow\Sdir,
\]
and a morphism from $p:X\to\Sdir$ to $q:Y\to\Sdir$ is a morphism $f:X\to Y$ of saturated directed spaces such that $qf=p$.
\ed

\bd
A clock map $p:X\to\Sdir$ is \textit{regular} if $p$ is regular as a map of
directed spaces.  Equivalently, for every $u\in\Path(X)$,
\[
  p\circ u \text{ constant}
  \quad\Longrightarrow\quad
  u \text{ constant}.
\]
We denote by
\[
  \RegClock\subseteq\Clock
\]
the full subcategory spanned by the regular clock maps.
\ed

By Lemma~\ref{lem:regular-local}, if $p:X\to\Sdir$ is a regular clock map and $u\in\Path(X)$, then whenever $p\circ u$ is constant on a subinterval, $u$ is constant on that subinterval.

\begin{prop}\label{prop:clock-lp}
The category $\Clock=\SDTop/\Sdir$ is locally presentable.
\end{prop}

\begin{proof}
A slice category of a locally presentable category is locally presentable. Since $\SDTop$ is locally presentable by Theorem~\ref{thm:sat-lp}, the slice $\SDTop/\Sdir$ is locally presentable.
\end{proof}

For each point $s\in \Circle$, let
\[
  \mathbf 1_s\longrightarrow \Sdir
\]
be the clock map selecting $s$.  Let
\[
  \vec I_s\longrightarrow \Sdir
\]
be the constant clock map from $\vec I$ to $s$.  There is a unique map in
$\Clock$
\[
  \rho_s:\vec I_s\longrightarrow \mathbf 1_s.
\]

\begin{prop}\label{prop:orthogonality}
A clock map $p:X\to\Sdir$ is regular if and only if it is orthogonal to $\rho_s$ for every $s\in \Circle$.  Hence
\[
  \RegClock
  =\{\rho_s\mid s\in \Circle\}^{\perp}
\]
as a full subcategory of $\Clock$.
\end{prop}

\begin{proof}
For a clock map $p:X\to\Sdir$, a morphism $\mathbf 1_s\to p$ in $\Clock$ is precisely a state $x\in\abs X$ with $p(x)=s$.  A morphism $\vec I_s\to p$ in $\Clock$ is precisely a directed path $u\in\Path(X)$ such that $p\circ u$ is the constant path at $s$. Under precomposition with $\rho_s$, the state $x$ is sent to the constant path at $x$. The map
\[
  \Clock(\mathbf 1_s,p)\longrightarrow \Clock(\vec I_s,p)
\]
is always injective, because distinct states define distinct constant paths. It is surjective precisely when every directed path lying over the constant clock value $s$ is constant.  Requiring this for all $s\in \Circle$ is exactly regularity of $p$.
\end{proof}

\begin{cor}\label{cor:regclock-lp}
The category $\RegClock$ is locally presentable.
\end{cor}

\begin{proof}
By Proposition~\ref{prop:clock-lp}, $\Clock$ is locally presentable.  By Proposition~\ref{prop:orthogonality}, $\RegClock$ is a small-orthogonality class in $\Clock$, because $\Circle$ has a set of points.  Therefore $\RegClock$ is locally presentable by \cite[Theorem~1.39]{TheBook}.
\end{proof}

Let
\[
U:\RegClock\longrightarrow \SDTop
\]
be the forgetful functor which sends a regular clock map $p:X\to \Sdir$ to its underlying saturated directed space $X$. This functor has a right adjoint, but no left adjoint. The right adjoint sends a saturated directed space $X$ to the regular clock map
\[
R(X)\longrightarrow \Sdir
\]
whose underlying continuous map is the first projection
\[
\pi_1:\Sdir\times X\longrightarrow \Sdir .
\]
The directed paths of $R(X)$ are the paths $(\lambda,u):[0,1]\to \Sdir\times X$ such that $\lambda$ is directed in $\Sdir$, $u$ is directed in $X$, and, for every interval $[a,b]\subseteq [0,1]$,
\[
\lambda \text{ constant on }[a,b]
\quad\Longrightarrow\quad
u \text{ constant on }[a,b].
\]
This directed structure is saturated, and the projection $\pi_1:R(X)\to \Sdir$ is regular by construction. If $p:Y\to \Sdir$ is a regular clock map, then every directed map $f:Y\to X$ induces a unique morphism of regular clock maps
\[
\widetilde f:Y\longrightarrow R(X),
\qquad
\widetilde f(y)=(p(y),f(y)).
\]
Indeed, if $v\in \vec P(Y)$, then
\[
\widetilde f v=(pv,fv),
\]
and whenever $pv$ is constant on an interval, regularity of $p$ implies that $v$, and hence $fv$, is constant on that interval. Conversely, every morphism $Y\to R(X)$ over $\Sdir$ is uniquely of this form by projection to the second factor. Hence there is a natural bijection
\[
\RegClock(Y,R(X))
\cong
\SDTop(U(Y),X),
\]
so $U\dashv R$. On the other hand, $U$ has no left adjoint in general, because a functor with a left adjoint must preserve limits, whereas $U$ does not preserve binary products. For example, the product of
\[
\operatorname{id}_{\Sdir}:\Sdir\to \Sdir
\]
with itself in $\RegClock$ is the pullback over $\Sdir$, hence is again $\Sdir\to\Sdir$. Its image under $U$ is $\Sdir$. But the product of the two underlying saturated directed spaces in $\SDTop$ is $\Sdir\times \Sdir$, which is not isomorphic to $\Sdir$. Therefore $U$ does not preserve products, and consequently it cannot have a left adjoint.

Let $J:\RegClock\hookrightarrow \Clock$ be the inclusion functor. It has both a left adjoint and a right adjoint. Let $p:X\to \Sdir$ be a clock map. The value $L(p)$ of the left adjoint is obtained by collapsing all vertical directed paths. More precisely, let $\sim_p$ be the smallest equivalence relation on $|X|$ such that
\[
u(s)\sim_p u(t)
\]
for every $u\in \vec P(X)$ with $p\circ u$ constant and every $s,t\in [0,1]$. Since $\sim_p$ is contained in the equivalence relation of belonging to the same fibre of $p$, the clock map $p$ factors uniquely as
\[
X \longrightarrow X/{\sim_p}
\xlongrightarrow{\overline p} \Sdir .
\]
The directed structure on $X/{\sim_p}$ is the saturated final directed structure induced by the quotient map $X\to X/{\sim_p}$. Then $\overline p:X/{\sim_p}\to \Sdir$ is regular, and this construction defines the left adjoint $L$. Indeed, every morphism $f:p\to q$ from $p$ to a regular clock map $q:Y\to \Sdir$ sends each vertical directed path of $X$ to a vertical directed path of $Y$, hence to a constant path; therefore $f$ is constant on the $\sim_p$-classes and factors uniquely through $X/{\sim_p}$. Thus
\[
\RegClock(L(p),q)
\cong
\Clock(p,J(q)).
\]
The value $R(p)$ of the right adjoint has the same underlying $\Delta$-generated space and the same clock map $p:|X|\to \Sdir$, but a smaller directed structure. Namely, a directed path of $R(p)$ is a directed path $u\in \vec P(X)$ such that, for every interval $[a,b]\subseteq [0,1]$,
\[
p\circ u \text{ constant on }[a,b]
\quad\Longrightarrow\quad
u \text{ constant on }[a,b].
\]
Equivalently, the directed paths of $R(p)$ are the \textit{locally $p$-regular} directed paths of $X$, i.e. those having no nonconstant subpath contained in a fibre of $p$. This defines a saturated directed structure, since local $p$-regularity contains constants and is stable under normalized composition, reparametrization, and saturation. The resulting clock map $R(p)\to \Sdir$ is regular. If $q:Y\to \Sdir$ is regular and $f:q\to p$ is a morphism of clock maps, then for every $v\in \vec P(Y)$, the path $f\circ v$ is locally $p$-regular because $p\circ f=q$ and $q$ is regular. Hence $f$ factors uniquely through $R(p)$, giving the natural bijection
\[
\RegClock(q,R(p))
\cong
\Clock(J(q),p).
\]
Thus $L\dashv J\dashv R$, where $L$ is the reflection obtained by quotienting vertical directed paths and $R$ is the coreflection obtained by retaining only locally clock-regular directed paths.

\section{Traces, clock lifts, and clock length}
\label{sec:clock-lift-length}

Let $p:X\to\Sdir$ be a regular clock map.  Thus $X$ is saturated.  Fix states $\alpha,\beta\in\abs X$.

\bd[Trace space]
Let $\Trace(X)(\alpha,\beta)$ be the quotient of
$\Path(X)(\alpha,\beta)$ by the equivalence relation generated by
\[
  u\sim u\circ\phi,
  \qquad \phi\in\Istar.
\]
The quotient map is denoted
\[
  q_{\alpha,\beta}:\Path(X)(\alpha,\beta)
  \longrightarrow
  \Trace(X)(\alpha,\beta).
\]
The quotient is taken in $\Top$.
\ed

Choose, once and for all, a real number $a\in\mathbb R$ such that
\[
  \pi(a)=p(\alpha).
\]
For each $u\in\Path(X)(\alpha,\beta)$, the path $p\circ u$ is directed in
$\Sdir$, hence admits a unique nondecreasing lift
\[
  \lambda_u:[0,1]\longrightarrow\mathbb R
\]
with $\lambda_u(0)=a$.  

\bd[Clock length of a directed path] \label{def:clock-length}
	Define the \textit{clock length} of $u$ by
\[
  L(u)=\lambda_u(1)-\lambda_u(0)=\lambda_u(1)-a.
\]
\ed

\begin{lem}\label{lem:lifts-continuous}
The assignment
\[
  u\longmapsto\lambda_u
\]
defines a continuous map
\[
  \Path(X)(\alpha,\beta)
  \longrightarrow
  \ttop([0,1],\mathbb R).
\]
Consequently $L:\Path(X)(\alpha,\beta)\to\mathbb R$ is continuous.
\end{lem}

\begin{proof}
The covering map $\pi:\mathbb R\to\mathbb R/\mathbb Z$ has the unique path-lifting property once the initial value is fixed.  The lifting operation is continuous for compact-open path spaces: locally on $[0,1]$, the path $p\circ u$ takes values in an evenly covered arc of $\Circle$, and the lift is obtained by composing with the corresponding local inverse of $\pi$.  Compactness of $[0,1]$ reduces the verification to finitely many such arcs.  This proves continuity for compact-open mapping spaces, hence also after $\Delta$-kelleyfication.  Continuity of $L$ follows by evaluation at $1$.
\end{proof}

\begin{lem}\label{lem:length-zero}
For $u\in\Path(X)(\alpha,\beta)$, one has $L(u)=0$ if and only if $u$
is constant.
\end{lem}

\begin{proof}
Since $\lambda_u$ is nondecreasing, $L(u)=0$ implies that $\lambda_u$ is constant.  Hence $p\circ u$ is constant.  Since $p$ is a regular clock map, $u$ is constant. Conversely, if $u$ is constant, then $p\circ u$ is constant, so its lift with initial value $a$ is constant.  Hence $L(u)=0$.
\end{proof}

\section{Clock normalization}
\label{sec:clock-normalization}

For $u\in\Path(X)(\alpha,\beta)$, define
\[
  \theta_u(t)=
  \begin{cases}
    \dfrac{\lambda_u(t)-a}{L(u)}, & L(u)>0,\\[1.2em]
    t, & L(u)=0.
  \end{cases}
\]
Then $\theta_u\in\Istar$.

\begin{lem}\label{lem:theta-continuous}
The assignment
\[
  u\longmapsto\theta_u
\]
defines a continuous map
\[
  \Path(X)(\alpha,\beta)\longrightarrow\Istar.
\]
\end{lem}

\begin{proof}
By Lemma~\ref{lem:lifts-continuous}, the maps $u\mapsto\lambda_u$ and $u\mapsto L(u)$ are continuous.  On $L^{-1}(]0,\infty[)$, the formula
\[
  \theta_u(t)=\frac{\lambda_u(t)-a}{L(u)}
\]
therefore gives a continuous map to $\Top([0,1],[0,1])$. It remains to check the gluing along $L^{-1}(0)$.  Choose a lift $b\in\mathbb R$ of $p(\beta)$.  For every $u$, the endpoint $\lambda_u(1)$ belongs to the discrete subset $b+\mathbb Z\subset\mathbb R$. Thus the possible values of $L(u)=\lambda_u(1)-a$ form a discrete subset of $[0,\infty[$.  Since $L$ is continuous, $L^{-1}(0)$ is both open and closed in $\Path(X)(\alpha,\beta)$.  On this open-and-closed subspace, $\theta_u=\Id_{[0,1]}$, so the two formulae glue continuously.
\end{proof}

\begin{lem}\label{lem:constant-on-fibres}
For every $u\in\Path(X)(\alpha,\beta)$, the path $u$ is constant on each fibre of $\theta_u$.
\end{lem}

\begin{proof}
If $L(u)=0$, then $u$ is constant by Lemma~\ref{lem:length-zero}.  Suppose $L(u)>0$.  If $\theta_u(s)=\theta_u(t)$ with $s\leq t$, then $\lambda_u(s)=\lambda_u(t)$.  Since $\lambda_u$ is nondecreasing, $\lambda_u$ is constant on $[s,t]$.  Hence $p\circ u$ is constant on $[s,t]$.  By Lemma~\ref{lem:regular-local}, $u$ is constant on $[s,t]$. Therefore $u(s)=u(t)$.
\end{proof}

\begin{lem}\label{lem:factor}
For every $u\in\Path(X)(\alpha,\beta)$, there exists a unique continuous path
\[
  u^{\natural}:[0,1]\longrightarrow\abs X
\]
such that
\[
  u=u^{\natural}\circ\theta_u.
\]
Moreover $u^{\natural}\in\Path(X)(\alpha,\beta)$.
\end{lem}

\begin{proof}
The map $\theta_u:[0,1]\to[0,1]$ is continuous and surjective.  Since $[0,1]$ is compact, $\theta_u$ is a quotient map.  By Lemma~\ref{lem:constant-on-fibres}, $u$ is constant on the fibres of $\theta_u$.  Therefore there exists a unique continuous path $u^{\natural}:[0,1]\to\abs X$ satisfying $u=u^{\natural}\circ\theta_u$. At this point we only know that $u^{\natural}$ is continuous.  Since $\theta_u\in\Istar$ and
\[
  u^{\natural}\circ\theta_u=u\in\Path(X),
\]
saturation of $X$ implies $u^{\natural}\in\Path(X)$.  The endpoint conditions follow from $\theta_u(0)=0$ and $\theta_u(1)=1$.
\end{proof}

\bd[Clock normalization] \label{def:naturalization-directed-path}
The directed path $u^{\natural}$ is called the \textit{clock-normalized representative} of the directed path $u$.
\ed

\begin{lem}\label{lem:natural-clock}
If $L(u)>0$, then the clock lift of $u^{\natural}$ is
\[
  \lambda_{u^{\natural}}(t)=a+L(u)t.
\]
If $L(u)=0$, then $u^{\natural}=u$ is the constant path at $\alpha$.
\end{lem}

\begin{proof}
If $L(u)=0$, the assertion follows from Lemma~\ref{lem:length-zero} and from $\theta_u=\Id_{[0,1]}$.  Suppose $L(u)>0$.  Since $u=u^{\natural}\circ\theta_u$, uniqueness of the lift with initial value $a$ gives
\[
  \lambda_u(t)=\lambda_{u^{\natural}}(\theta_u(t)).
\]
By definition of $\theta_u$,
\[
  \lambda_u(t)=a+L(u)\theta_u(t).
\]
Since $\theta_u$ is surjective,
\[
  \lambda_{u^{\natural}}(s)=a+L(u)s
\]
for every $s\in[0,1]$.
\end{proof}

Let
\[
  \Path^{\natural}(X)(\alpha,\beta)
  \subseteq
  \Path(X)(\alpha,\beta)
\]
be the subspace consisting of paths $v$ such that either $v$ is constant, or the lift of $p\circ v$ starting at $a$ is affine:
\[
  \lambda_v(t)=a+L(v)t.
\]
By Lemma~\ref{lem:natural-clock}, $u^{\natural}\in \Path^{\natural}(X)(\alpha,\beta)$ for every $u$.

\section{Continuity, deformation retraction, and traces}
\label{sec:deformation-retract}

We shall use the following continuity lemma for monotone quotient maps.

\begin{lem}\label{lem:quotient-continuity}
Let $Y$ be a $\Delta$-generated space.  Let $A$ be the subspace of $\ttop([0,1],Y)\times\Istar$ consisting of pairs $(f,\theta)$ such that $f$ is constant on every fibre of $\theta$.  For each $(f,\theta)\in A$, let $\overline f$ be the unique continuous path satisfying
\[
  f=\overline f\circ\theta.
\]
Then
\[
  (f,\theta)\longmapsto\overline f
\]
defines a continuous map
\[
  A\longrightarrow\ttop([0,1],Y).
\]
\end{lem}

\begin{proof}
Let
\[
Q:A\longrightarrow \ttop([0,1],Y),
\qquad
Q(f,\theta)=\overline f .
\]
By the exponential law in the category of $\Delta$-generated spaces, it is enough to prove that the adjoint map
\[
\widetilde Q:A\times [0,1]\longrightarrow Y,
\qquad
\widetilde Q(f,\theta,t)=\overline f(t),
\]
is continuous. Fix a point $(f,\theta,t)\in A\times[0,1]$, and let $U\subseteq Y$ be an open neighborhood of $\overline f(t)$. Put
\[
F=\theta^{-1}(t).
\]
Since $\theta$ is continuous and nondecreasing, $F$ is a nonempty compact interval. Since $f=\overline f\theta$, the map $f$ is constant on $F$, and
\[
f(F)=\{\overline f(t)\}\subseteq U.
\]
Thus $F\subseteq f^{-1}(U)$. Choose an open neighborhood $W$ of $F$ in $[0,1]$ such that
\[
\overline W\subseteq f^{-1}(U).
\]
Equivalently,
\[
f(\overline W)\subseteq U.
\]
We shall now construct explicitly a neighborhood of $(\theta,t)$ such that the fibre of the nearby map over the nearby value remains contained in $W$. Write
\[
F=[a,b].
\]
We first consider the case
\[
0<a\leq b<1.
\]
Choose $r,s\in[0,1]$ such that
\[
r<a\leq b<s
\qquad\text{and}\qquad
F\subseteq ]r,s[\subseteq [r,s]\subseteq W.
\]
Since $F=\theta^{-1}(t)$ and $\theta$ is nondecreasing, we have
\[
\theta(r)<t<\theta(s).
\]
Choose
\[
0<\varepsilon<
\frac{1}{3}\min\{\,t-\theta(r),\,\theta(s)-t\,\}.
\]
Let
\[
\mathcal U_\theta
=
\left\{
\theta'\in \Istar
\ \middle|\
\theta'(r)\in]\theta(r)-\varepsilon,\theta(r)+\varepsilon[,
\ \theta'(s)\in]\theta(s)-\varepsilon,\theta(s)+\varepsilon[
\right\}.
\]
This is an open neighborhood of $\theta$ in the compact-open topology on $\Istar$. Let
\[
J=]t-\varepsilon,t+\varepsilon[\cap[0,1].
\]
If $\theta'\in\mathcal U_\theta$ and $t'\in J$, then
\[
\theta'(r)<t'<\theta'(s).
\]
Indeed,
\[
\theta'(r)<\theta(r)+\varepsilon<t-2\varepsilon<t'
\]
and similarly
\[
t'<t+\varepsilon<\theta(s)-2\varepsilon<\theta'(s).
\]
Since $\theta'$ is nondecreasing, no point $x\leq r$ can satisfy $\theta'(x)=t'$, and no point $x\geq s$ can satisfy $\theta'(x)=t'$. Hence
\[
(\theta')^{-1}(t')\subseteq ]r,s[\subseteq W .
\]
Now consider the endpoint case
\[
F=[0,b],
\qquad b<1.
\]
Choose $s>b$ such that
\[
[0,s]\subseteq W.
\]
Then $\theta(s)>t$. Choose
\[
0<\varepsilon<\frac{1}{3}(\theta(s)-t),
\]
put
\[
\mathcal U_\theta
=
\left\{
\theta'\in \Istar
\ \middle|\
\theta'(s)\in]\theta(s)-\varepsilon,\theta(s)+\varepsilon[
\right\},
\]
and put
\[
J=]t-\varepsilon,t+\varepsilon[\cap[0,1].
\]
For every $\theta'\in\mathcal U_\theta$ and every $t'\in J$, one has
\[
t'<\theta'(s).
\]
Therefore, by monotonicity of $\theta'$, no point $x\geq s$ belongs to $(\theta')^{-1}(t')$. Hence
\[
(\theta')^{-1}(t')\subseteq [0,s[\subseteq W .
\]
The other endpoint case is similar. If
\[
F=[a,1],
\qquad a>0,
\]
choose $r<a$ such that
\[
[r,1]\subseteq W.
\]
Then $\theta(r)<t$. Choose
\[
0<\varepsilon<\frac{1}{3}(t-\theta(r)),
\]
put
\[
\mathcal U_\theta
=
\left\{
\theta'\in \Istar
\ \middle|\
\theta'(r)\in]\theta(r)-\varepsilon,\theta(r)+\varepsilon[
\right\},
\]
and again put
\[
J=]t-\varepsilon,t+\varepsilon[\cap[0,1].
\]
For every $\theta'\in\mathcal U_\theta$ and every $t'\in J$, one has
\[
\theta'(r)<t'.
\]
Therefore no point $x\leq r$ belongs to $(\theta')^{-1}(t')$, and hence
\[
(\theta')^{-1}(t')\subseteq ]r,1]\subseteq W .
\]
Finally, if $F=[0,1]$, then we may replace $W$ by $[0,1]$, and there is nothing to prove: for every $\theta'\in\Istar$ and every $t'\in[0,1]$, one has
\[
(\theta')^{-1}(t')\subseteq W.
\]
In all cases we have obtained an open neighborhood
\[
\mathcal U_\theta\subseteq \Istar
\]
of $\theta$ and an open neighborhood
\[
J\subseteq [0,1]
\]
of $t$ such that, for all $\theta'\in\mathcal U_\theta$ and all
$t'\in J$,
\[
(\theta')^{-1}(t')\subseteq W.
\]
Now define an open neighborhood of $f$ in the compact-open topology by
\[
\mathcal U_f
=
\{\,f'\in \ttop([0,1],Y)\mid f'(\overline W)\subseteq U\,\}.
\]
Since $\overline W$ is compact in $[0,1]$, this is a compact-open subbasic neighborhood of $f$. The binary product in $\Top$ is obtained by applying $\Delta$-kelleyfication to the ordinary topological product. Since $\Delta$-kelleyfication only refines the topology, every ordinary product-open subset is open in the $\Delta$-generated product. Therefore \[\mathcal U_f\times \mathcal U_\theta\] is open in $\ttop([0,1],Y)\times \Istar$, and \[\bigl((\mathcal U_f\times\mathcal U_\theta)\cap A\bigr)\times J\] is open in $A\times[0,1]$. We claim that
\[
\bigl((\mathcal U_f\times \mathcal U_\theta)\cap A\bigr)\times J
\]
is mapped by $\widetilde Q$ into $U$. Let
\[
(f',\theta',t')\in
\bigl((\mathcal U_f\times \mathcal U_\theta)\cap A\bigr)\times J.
\]
Since $(f',\theta')\in A$, the map $f'$ is constant on each fibre of $\theta'$, and there exists a unique continuous path $\overline{f'}$ such that
\[
f'=\overline{f'}\theta'.
\]
Choose $s'\in(\theta')^{-1}(t')$. By construction,
\[
s'\in W\subseteq \overline W.
\]
Therefore
\[
\overline{f'}(t')
=
f'(s')
\in
f'(\overline W)
\subseteq U.
\]
Thus $\widetilde Q$ is continuous at $(f,\theta,t)$. Since this point was arbitrary, $\widetilde Q$ is continuous. By the exponential law, $Q$ is continuous. This proves continuity for compact-open mapping spaces, and therefore also after $\Delta$-kelleyfication, since $A$ is $\Delta$-generated. 
\end{proof}

\begin{prop}\label{prop:normalization-continuous}
The clock-normalization assignment
\[
  N:\Path(X)(\alpha,\beta)
  \longrightarrow
  \Path^{\natural}(X)(\alpha,\beta),
  \qquad
  N(u)=u^{\natural},
\]
is continuous.
\end{prop}

\begin{proof}
By Lemma~\ref{lem:theta-continuous}, the map
\[
  u\longmapsto (u,\theta_u)
\]
from $\Path(X)(\alpha,\beta)$ to $\ttop([0,1],\abs X)\times\Istar$ is continuous.  By Lemma~\ref{lem:constant-on-fibres}, its image lies in the subspace $A$ of Lemma~\ref{lem:quotient-continuity}.  Applying Lemma~\ref{lem:quotient-continuity} gives continuity of $u\mapsto u^{\natural}$ as a map to $\ttop([0,1],\abs X)$.  By Lemma~\ref{lem:factor}, the resulting path is directed, and by Lemma~\ref{lem:natural-clock} it belongs to $\Path^{\natural}(X)(\alpha,\beta)$.
\end{proof}

For $s\in[0,1]$, define
\[
  \psi_{u,s}(t)=(1-s)\theta_u(t)+st.
\]
Then $\psi_{u,s}\in\Istar$.  Define
\[
  H:\Path(X)(\alpha,\beta)\times[0,1]
  \longrightarrow
  \Path(X)(\alpha,\beta)
\]
by
\[
  H(u,s)(t)=u^{\natural}(\psi_{u,s}(t)).
\]

\begin{prop}\label{prop:deformation}
The map $H$ is a continuous homotopy such that
\[
  H(u,0)=u,
  \qquad
  H(u,1)=u^{\natural}.
\]
Moreover, if $u\in\Path^{\natural}(X)(\alpha,\beta)$, then $H(u,s)=u$ for every $s\in[0,1]$.  Hence $\Path^{\natural}(X)(\alpha,\beta)$ is a strong deformation retract of $\Path(X)(\alpha,\beta)$.
\end{prop}

\begin{proof}
The map $(u,s)\mapsto\psi_{u,s}$ is continuous because $u\mapsto\theta_u$ is continuous.  Composition of paths is continuous in $\Top$, and $u\mapsto u^{\natural}$ is continuous by Proposition~\ref{prop:normalization-continuous}; hence $H$ is continuous. For every $(u,s)$, the path $H(u,s)=u^{\natural}\circ\psi_{u,s}$ is directed because $u^{\natural}$ is directed and directed paths are closed under reparametrization.  The endpoints are preserved because $\psi_{u,s}(0)=0$ and $\psi_{u,s}(1)=1$. For $s=0$, $\psi_{u,0}=\theta_u$, so
\[
  H(u,0)=u^{\natural}\circ\theta_u=u.
\]
For $s=1$, $\psi_{u,1}=\Id_{[0,1]}$, so
\[
  H(u,1)=u^{\natural}.
\]
Finally, if $u$ is clock-normalized, then $\theta_u=\Id_{[0,1]}$, and therefore $\psi_{u,s}=\Id_{[0,1]}$ for every $s$.  Thus $H(u,s)=u$.
\end{proof}

\begin{lem}\label{lem:normalization-class}
For every $u\in\Path(X)(\alpha,\beta)$ and every $\phi\in\Istar$, one has
\[
  (u\circ\phi)^{\natural}=u^{\natural}.
\]
Consequently $N(u)=u^{\natural}$ is constant on trace classes.
\end{lem}

\begin{proof}
If $L(u)=0$, then $u$ is constant, and the claim is immediate.  Suppose $L(u)>0$.  The clock lift of $p\circ u\circ\phi$ starting at $a$ is $\lambda_u\circ\phi$.  Hence $L(u\circ\phi)=L(u)$, and
\[
  \theta_{u\circ\phi}(t)
  =\frac{\lambda_u(\phi(t))-a}{L(u)}
  =\theta_u(\phi(t)).
\]
Since $u=u^{\natural}\circ\theta_u$, we get
\[
  u\circ\phi
  =u^{\natural}\circ\theta_u\circ\phi
  =u^{\natural}\circ\theta_{u\circ\phi}.
\]
By uniqueness of the factorization through $\theta_{u\circ\phi}$, the clock-normalized representative of $u\circ\phi$ is $u^{\natural}$.  Since the trace relation is generated by the elementary identifications $u\sim u\circ\phi$, the final assertion follows.
\end{proof}

\begin{lem}\label{lem:natural-unique}
Every trace class contains exactly one clock-normalized path.
\end{lem}

\begin{proof}
Existence follows from Lemma~\ref{lem:factor}: the path $u^{\natural}$ is clock-normalized and lies in the trace class of $u$, because $u=u^{\natural}\circ\theta_u$. For uniqueness, suppose $v,w\in\Path^{\natural}(X)(\alpha,\beta)$ belong to the same trace class.  By Lemma~\ref{lem:normalization-class}, the normalization map is constant on trace classes.  Therefore
\[
  v=N(v)=N(w)=w,
\]
because normalized paths are fixed by $N$.
\end{proof}

\begin{prop}\label{prop:N-homeo-Trace}
The restriction of the quotient map
\[
  q_{\alpha,\beta}\big|_{\Path^{\natural}(X)(\alpha,\beta)}:
  \Path^{\natural}(X)(\alpha,\beta)
  \longrightarrow
  \Trace(X)(\alpha,\beta)
\]
is a homeomorphism.
\end{prop}

\begin{proof}
By Lemma~\ref{lem:natural-unique}, the restriction is bijective.  It is continuous because $q_{\alpha,\beta}$ is continuous. By Lemma~\ref{lem:normalization-class}, the continuous map
\[
  N:\Path(X)(\alpha,\beta)
  \longrightarrow
  \Path^{\natural}(X)(\alpha,\beta)
\]
is constant on the equivalence classes defining $\Trace(X)(\alpha,\beta)$. Therefore the universal property of the quotient gives a unique continuous map
\[
  \overline N:\Trace(X)(\alpha,\beta)
  \longrightarrow
  \Path^{\natural}(X)(\alpha,\beta)
\]
such that $\overline N\circ q_{\alpha,\beta}=N$.  This map is inverse to $q_{\alpha,\beta}|_{\Path^{\natural}(X)(\alpha,\beta)}$, by Lemma~\ref{lem:natural-unique}.  Hence the restriction of the quotient map is a homeomorphism.
\end{proof}

\begin{thm}\label{thm:main}
Let $p:X\to\Sdir$ be a regular clock map, with $X$ saturated.  Then, for every pair of states $\alpha,\beta\in\abs X$, the canonical quotient map
\[
  q_{\alpha,\beta}:\Path(X)(\alpha,\beta)
  \longrightarrow
  \Trace(X)(\alpha,\beta)
\]
is a homotopy equivalence.  In particular, it is a weak homotopy equivalence.
\end{thm}

\begin{proof}
By Proposition~\ref{prop:deformation}, the inclusion
\[
  i:\Path^{\natural}(X)(\alpha,\beta)
  \hookrightarrow
  \Path(X)(\alpha,\beta)
\]
is a strong deformation retract, with retraction $N(u)=u^{\natural}$.  By Proposition~\ref{prop:N-homeo-Trace}, the restriction
\[
  q_{\alpha,\beta}|_{\Path^{\natural}(X)(\alpha,\beta)}:
  \Path^{\natural}(X)(\alpha,\beta)
  \longrightarrow
  \Trace(X)(\alpha,\beta)
\]
is a homeomorphism. Let
\[
  j:\Trace(X)(\alpha,\beta)
  \longrightarrow
  \Path(X)(\alpha,\beta)
\]
be the composite of the inverse of this homeomorphism with the inclusion $i$.  Then
\[
  q_{\alpha,\beta}\circ j=\Id_{\Trace(X)(\alpha,\beta)}.
\]
On the other hand,
\[
  j\circ q_{\alpha,\beta}=N.
\]
The homotopy $H$ from Proposition~\ref{prop:deformation} gives
\[
  \Id_{\Path(X)(\alpha,\beta)}\simeq N=j\circ q_{\alpha,\beta}.
\]
Thus $q_{\alpha,\beta}$ is a homotopy equivalence.
\end{proof}

\begin{rem}
Saturation is used in exactly one essential place: Lemma~\ref{lem:factor}. The clock-normalization construction first produces a continuous path $u^{\natural}$ satisfying $u=u^{\natural}\circ\theta_u$.  Since $u$ is directed and $\theta_u\in\Istar$, saturation is what implies that $u^{\natural}$ is itself directed.  The rest of the proof then uses this fact repeatedly.
\end{rem}

\appendix

\section{A length function need not come from a regular clock map}
\label{sec:stronger}

The existence of a regular clock map is strictly stronger than the existence of a regular additive endpoint-reparametrization-invariant length function. Consider the closed disk $\Disk$, and identify its boundary with the directed circle $\Sdir$. Define a directed space $X$ by taking $|X|=\Disk$, and by declaring the directed paths of $X$ to be the constant paths in $\Disk$, together with the directed paths whose image is contained in the boundary $\partial \Disk\simeq \Sdir$. This directed structure is saturated: if $u\phi$ is directed for some endpoint-preserving $\phi\in\mathcal I$, then either $u\phi$ is constant, in which case $u$ is constant since $\phi$ is surjective, or $u\phi$ lies in the boundary, in which case $u$ also lies in the boundary and the assertion follows from the saturation of the usual directed circle. Define
\[
L(u)=0
\]
for constant paths, and, for a nonconstant boundary-directed path $u$, define
\[
L(u)=\widetilde u(1)-\widetilde u(0),
\]
where $\widetilde u:[0,1]\to\mathbb R$ is any nondecreasing lift of $u$ along $\mathbb R\to \Circle$. This is independent of the chosen lift. It is the usual length on the directed circle, extended by zero on constant paths. Hence $L$ is additive ($L(u*_Nv)=L(u)+L(v)$), endpoint-reparametrization invariant ($L(u\phi)=L(u)$ for $\phi\in \Istar$), continuous for the $\Delta$-kelleyfication of the compact-open topology, and regular ($L(u)=0$ if and only if $u$ is constant). Nevertheless, $X$ admits no regular clock map $p:X\to\Sdir$. Indeed, if such a map existed, then its restriction
\[
p|_{\partial \Disk}:\Sdir\longrightarrow \Sdir
\]
would be a directed map which is nonconstant on every nonconstant directed boundary path, by regularity. Choose a lift $F:\mathbb R\to\mathbb R$ of $p|_{\partial \Disk}$. Since $p|_{\partial \Disk}$ is directed, $F$ is nondecreasing and satisfies
\[
F(t+1)=F(t)+n
\]
for some integer $n\geq 0$, the degree of $p|_{\partial \Disk}$. If $n=0$, then $F$ is nondecreasing and $1$-periodic, hence constant, contradicting regularity on the boundary. Thus $n\geq 1$. But $p|_{\partial \Disk}$ extends continuously to $\Disk$, since $p$ is defined on all of $X$. Any map $\Circle\to \Circle$ extending over the disk has degree $0$, a contradiction. Therefore $X$ has a regular additive endpoint-reparametrization-invariant length function, but no regular clock map.


\begin{thebibliography}{10}
	
	\bibitem{topologicalcat}
	J.~Ad{\'a}mek, H.~Herrlich, and G.~E. Strecker.
	\newblock Abstract and concrete categories: the joy of cats.
	\newblock {\em Repr. Theory Appl. Categ.}, (17):1--507, 2006.
	\newblock Reprint of the 1990 original [Wiley, New York; MR1051419].
	
	\bibitem{TheBook}
	J.~Ad{\'a}mek and J.~Rosick{\'y}.
	\newblock {\em Locally presentable and accessible categories}.
	\newblock Cambridge University Press, Cambridge, 1994.
	\newblock \href {https://doi.org/10.1017/cbo9780511600579.004}
	{\path{https://doi.org/10.1017/cbo9780511600579.004}}.
	
	\bibitem{reparam}
	U.~Fahrenberg and M.~Raussen.
	\newblock Reparametrizations of continuous paths.
	\newblock {\em J. Homotopy Relat. Struct.}, 2(2):93--117, 2007.
	
	\bibitem{DAT_book}
	L.~Fajstrup, E.~Goubault, E.~Haucourt, S.~Mimram, and M.~Raussen.
	\newblock {\em Directed algebraic topology and concurrency. {With} a foreword
		by {Maurice} {Herlihy} and a preface by {Samuel} {Mimram}}.
	\newblock SpringerBriefs Appl. Sci. Technol. Springer, 2016.
	\newblock \href {https://doi.org/10.1007/978-3-319-15398-8}
	{\path{https://doi.org/10.1007/978-3-319-15398-8}}.
	
	\bibitem{FR}
	L.~Fajstrup and J.~Rosick{\'y}.
	\newblock A convenient category for directed homotopy.
	\newblock {\em Theory Appl. Categ.}, 21:7--20, 2008.
	
	\bibitem{mdtop}
	P.~Gaucher.
	\newblock Homotopical interpretation of globular complex by multipointed
	d-space.
	\newblock {\em Theory Appl. Categ.}, 22(22):588--621, 2009.
	
	\bibitem{leftproperflow}
	P.~Gaucher.
	\newblock Left properness of flows.
	\newblock {\em Theory Appl. Categ.}, 37(19):562--612, 2021.
	
	\bibitem{DirectedDegeneracy}
	P.~Gaucher.
	\newblock Directed degeneracy maps for precubical sets.
	\newblock {\em Theory Appl. Categ.}, 41(7):194--237, 2024.
	
	\bibitem{Moore3}
	P.~Gaucher.
	\newblock Homotopy theory of {M}oore flows ({III}).
	\newblock {\em North-West. Eur. J. Math.}, 10:55--113, 2024.
	
	\bibitem{GlobularNaturalSystem}
	P.~Gaucher.
	\newblock Natural homotopy of multipointed d-spaces.
	\newblock {\em Math. J. Okayama Univ.}, 68:13--62, 2026.
	
	\bibitem{ThickCubes}
	P.~Gaucher.
	\newblock Towards a theory of natural directed paths.
	\newblock {\em {Compositionality}}, 7(6), 2026.
	\newblock \href {https://doi.org/10.46298/compositionality-7-6}
	{\path{https://doi.org/10.46298/compositionality-7-6}}.
	
	\bibitem{mg}
	M.~Grandis.
	\newblock Directed homotopy theory. {I}.
	\newblock {\em Cah. Topol. G\'eom. Diff\'er. Cat\'eg.}, 44(4):281--316, 2003.
	
	\bibitem{henry2020minimal}
	S.~Henry.
	\newblock Minimal model structures, 2020.
	\newblock \href {https://doi.org/10.48550/arXiv.2011.13408}
	{\path{https://doi.org/10.48550/arXiv.2011.13408}}.
	
	\bibitem{zbMATH06404305}
	A.~Hirschowitz, M.~Hirschowitz, and T.~Hirschowitz.
	\newblock Saturating directed spaces.
	\newblock {\em J. Homotopy Relat. Struct.}, 9(2):273--283, 2014.
	\newblock \href {https://doi.org/10.1007/s40062-013-0025-8}
	{\path{https://doi.org/10.1007/s40062-013-0025-8}}.
	
	\bibitem{MR2545830}
	S.~Krishnan.
	\newblock A convenient category of locally preordered spaces.
	\newblock {\em Appl. Categ. Structures}, 17(5):445--466, 2009.
	\newblock \href {https://doi.org/10.1007/s10485-008-9140-9}
	{\path{https://doi.org/10.1007/s10485-008-9140-9}}.
	
	\bibitem{MR2521708}
	M.~Raussen.
	\newblock Trace spaces in a pre-cubical complex.
	\newblock {\em Topology Appl.}, 156(9):1718--1728, 2009.
	\newblock \href {https://doi.org/10.1016/j.topol.2009.02.003}
	{\path{https://doi.org/10.1016/j.topol.2009.02.003}}.
	
	\bibitem{Ziemiaski2012}
	K.~Ziemia\'{n}ski.
	\newblock Categories of directed spaces.
	\newblock {\em {Fund. Math.}}, 217(1):55--71, 2012.
	\newblock \href {https://doi.org/10.4064/fm217-1-5}
	{\path{https://doi.org/10.4064/fm217-1-5}}.
	
\end{thebibliography}

\end{document}